%% file: ECC.tex
\documentclass[11pt, letterpaper]{amsart}

\usepackage[letterpaper, margin=4cm, headsep=25pt]{geometry}

\usepackage{amsmath,amssymb,amscd,amsthm,indentfirst}
\usepackage{amsfonts}
\usepackage[mathscr]{eucal}
\usepackage{enumerate}

\usepackage{booktabs}

\usepackage[small,nohug,heads=LaTeX]{diagrams}
\diagramstyle[labelstyle=\scriptstyle]
\usepackage{graphicx}
\usepackage{pstricks}

\newtheorem{MainThm}{Theorem}

\newtheorem{thm}{Theorem}[section]
\newtheorem{cor}[thm]{Corollary}
\newtheorem{lem}[thm]{Lemma}
\newtheorem{prop}[thm]{Proposition}

\theoremstyle{definition}
\newtheorem{defn}[thm]{Definition}
\theoremstyle{remark}
\newtheorem{rem}[thm]{Remark}

\newtheorem{example}[thm]{Example}
\numberwithin{equation}{section}

\newcommand{\bR}{\mathbb{R}}

\newcommand{\bZ}{\mathbb{Z}}

\newcommand{\MT}[2]{\bold{MT #1}(#2)}
\newcommand{\MTtheta}{\bold{MT}\theta}

\newcommand\lra{\longrightarrow}
\newcommand\lla{\longleftarrow}

\newcommand\Diff{\mathrm{Diff}}
\newcommand\Emb{\mathrm{Emb}}
\newcommand\Bun{\mathrm{Bun}}
\newcommand\map{\mathrm{map}}
\newcommand\Th{\mathrm{Th}}
\newcommand\colim{\mathrm{colim \,}}
\newcommand\hocolim{\mathrm{hocolim \,}}

\newarrow {Onto} ----{>>}
\newarrow {Equals} ===={=}

\newcommand{\CircNum}[1]{\ooalign{\hfil\raise .00ex\hbox{\scriptsize #1}\hfil\crcr\mathhexbox20D}}

\newcommand{\X}{\mathbf{X}}
\newcommand{\Gr}{\mathrm{Gr}}

\usepackage[hypertexnames=false]{hyperref}

\title[Embedded Cobordism Categories]{Embedded Cobordism Categories\\ and Spaces of Manifolds}
\author{Oscar Randal-Williams}
\thanks{Supported by an EPSRC Studentship, DTA grant number EP/P502667/1}
\email{randal-w@maths.ox.ac.uk}

\address{Mathematical Institute\\
	24-29 St Giles'\\
	Oxford\\
	OX1 3LB\\
	United Kingdom}
\date{\today}

\begin{document}

\begin{abstract}
Galatius, Madsen, Tillmann and Weiss \cite{GMTW} have identified the homotopy type of the classifying space of the cobordism category with objects $(d-1)$-dimensional manifolds embedded in $\bR^\infty$. In this paper we apply the techniques of spaces of manifolds, as developed by the author and Galatius in \cite{GR-W}, to identify the homotopy type of the cobordism category with objects $(d-1)$-dimensional submanifolds of a fixed background manifold $M$.

There is a description in terms of a space of sections of a bundle over $M$ associated to its tangent bundle. This can be interpreted as a form of Poincar\'{e} duality, relating a space of submanifolds of $M$ to a space of functions on $M$.
\end{abstract}
\maketitle

\input{chap1.tex}
\input{chap2.tex}
\input{chap3.tex}
\input{chap4.tex}
\input{chap5.tex}
\input{chap6.tex}
\input{chap7.tex}

\input{chap8.tex}

\bibliographystyle{amsalpha}
\bibliography{ECC}

\end{document}

%% file: chap1.tex
\section{Introduction}

Given a smooth manifold $M$, one can ask for a space $X$ of compact $(d-1)$-dimensional submanifolds of $M$, where paths in $X$ correspond to $d$-dimensional cobordisms in $[0,1] \times M$. One construction of such a space is as the classifying space of the \textit{cobordism category} $\mathcal{C}_d(M)$ having objects compact $(d-1)$-submanifolds of $M$ and morphisms cobordisms in $[0,1] \times M$. Galatius, Madsen, Tillmann and Weiss \cite{GMTW} have identified the homotopy type of this space when $M=\bR^\infty$, as the infinite loop space $\Omega^{\infty-1}\MT{O}{d}$ corresponding to a certain Thom spectrum.

In fact, given a \textit{tangential structure}, which is just a Serre fibration $\theta : \X \to \Gr_d(\bR^\infty)$, one can define the cobordism category $\mathcal{C}_\theta(M)$ using submanifolds equipped with $\theta$-structure. The above authors also show that $B\mathcal{C}_\theta(\bR^\infty) \simeq \Omega^{\infty-1}\MTtheta$ where $\MTtheta := \Th(-\theta^*\gamma_d \to \X)$ is the Thom spectrum of the complement to the rank $d$ vector bundle classified by $\theta$.

To state our main theorem, let us define a functor
$$T_\theta : \{\text{f.d.\ vector subspaces of $\bR^\infty$ and isomorphisms} \} \lra \mathbf{Top}_*$$
by $T_\theta(V) := \Th(\gamma_d^\perp \to \Gr_d^\theta(\bR \oplus V))$, where $\Gr_d^\theta(\bR \oplus V)$ is defined by the fibre product
\begin{diagram}
\Gr_d^\theta(\bR \oplus V) & \rTo & \X \\
\dTo & & \dTo^\theta \\
\Gr_d(\bR \oplus V) & \rTo & \Gr_d(\bR \oplus \bR^\infty).
\end{diagram}

\begin{MainThm}\label{MainTheorem}
Let $\theta: \X \to \Gr_d(\bR^\infty)$ be a Serre fibration and $M$ be a smooth $n$-manifold that is not necessarily compact and possibly has boundary. Then there is a weak homotopy equivalence
$$B \mathcal{C}_\theta(M) \simeq \Gamma_c(T_\theta^{fib}(TM) \to M)$$
where $\Gamma_c$ denotes the space of compactly supported sections, and the functor $T_\theta$ is applied fibrewise to the tangent bundle, by $Fr(M) \times_{GL_n(\bR)} T_\theta(\bR^n)$.
\end{MainThm}

In particular, taking the background manifold to be $n$-dimensional Euclidean space we obtain the weak homotopy equivalence
\begin{equation}\label{eq:FiniteEuclideanSpaceThm}
B \mathcal{C}_\theta(\bR^n) \simeq \Omega^n \Th(\gamma_d^\perp \to \Gr_d^\theta(\bR^{n+1})).
\end{equation}
This proves Theorem 6.3 of \cite{galatius-2006} in the more general case when the manifolds involved are equipped with a $\theta$-structure. Ayala \cite[Theorem 1.3.1]{Ayala} has also recently obtained the equivalence (\ref{eq:FiniteEuclideanSpaceThm}) by different means, and it also follows from work of Galatius and the author \cite[Theorem 3.12]{GR-W}.

\vspace{2ex}

Let us discuss some applications of this theorem. Given a functor $F: \mathbf{Man} \to \mathbf{Top}$ from the category of manifolds and closed embeddings to topological spaces, one may ask how far it is from being homotopy invariant: that is, what is the initial homotopy invariant functor $F^h$ with a natural transformation $F \to F^h$. Such a functor need not exist, but if it does its value on a manifold is unique up to homotopy equivalence. Considering $B\mathcal{C}_\theta(-)$ as a functor from manifolds to spaces, in \S\ref{sec:HomotopyInvariance} we show that its homotopy invariant approximation is a natural transformation
\begin{equation}\label{eq:CategoryStabilisation}
B\mathcal{C}_\theta(M) \lra \Omega^{\infty-1} \MTtheta \wedge M_+,
\end{equation}
where the right hand functor is equivalent to $\hocolim B\mathcal{C}_\theta(\bR^n \times M)$. Thus $\MTtheta$-bordism is the best homotopy invariant approximation to $B\mathcal{C}_\theta(-)$.

There is a stabilisation map $S^n \wedge T_\theta(V) \to T_\theta(\bR^n \oplus V)$ adjoint to a map $T_\theta(V) \to \Omega^n T_\theta(\bR^n \oplus V)$. This may be applied fibrewise to obtain a stabilisation map
\begin{equation}\label{eq:SectionsStabilisation}
\Gamma_c(T_\theta^{fib}(TM) \to M) \lra \Gamma_c(\Omega^{\infty-1} \MTtheta \wedge_{fib} TM^+ \to M),
\end{equation}
where the right hand functor is equivalent to $\hocolim \Gamma_c(\Omega^nT_\theta^{fib}(\bR^n \oplus TM) \to M)$. This is the space of sections of a parametrised spectrum \cite[p. 25]{CK}, and is interpreted as the $\MTtheta$-cohomology of $M$ twisted by the tangent bundle of $M$. In \S\ref{sec:PoincareDuality} we show that the stabilisation maps (\ref{eq:CategoryStabilisation}) and (\ref{eq:SectionsStabilisation}) convert Theorem \ref{MainTheorem} into Poincar\'{e} duality in $\MTtheta$-theory: they identify $\Omega^{\infty-1} \MTtheta \wedge M_+$ with $\Gamma_c(\Omega^{\infty-1} \MTtheta \wedge_{fib} TM^+ \to M)$. Thus one may think of Theorem \ref{MainTheorem} as an unstable refinement of Poincar\'{e} duality.

\vspace{2ex}

McDuff \cite{McDuff} has introduced a space $C^\pm(M)$ of configurations of positive and negative particles in a manifold $M$. In \S\ref{sec:ConfigurationsOfSignedPoints} we study the category $\mathcal{C}_1^+$ with objects signed configurations of points in $M$ and morphisms oriented 1-manifolds in $[0,1] \times M$, and ask how far it is from being the category of smooth paths in the space $C^\pm(M)$. We show that it cannot reasonably be considered to be a category of paths (the topology on spaces of paths in $C^\pm(M)$ is far coarser than that on spaces of morphisms in $\mathcal{C}_1^+$, even though they may be taken to be bijective), but nonetheless we show that $B\mathcal{C}_1^+(M)$ does have the homotopy type of $C^\pm(M)$. We also speculate as to what the analogue of this statement should be for manifolds of positive dimension.

\vspace{2ex}

In \S\ref{sec:ConfigurationSpaceMonoid} we show how Theorem \ref{MainTheorem} may be used to recover a theorem of Segal \cite{SegalConfigurationSpace} on the homotopy type of the classifying space of the labelled configuration space monoid $C(\bR \times M;X)$. The homotopy invariance results then imply that the homotopy invariant approximation to $BC(-;X)$ is $Q(\Sigma-_+ \wedge X_+)$, the homology theory represented by the suspension spectrum of the space of labels $X$.

In \S\ref{sec:BraidMonoid} we show how the methods of this paper (though not the precise statement of Theorem \ref{MainTheorem}) may be used to study the monoid $\mathcal{B} := \coprod_n B\beta_n$ of classifying spaces of braid groups, where the product is by disjoint union of braids. This monoid may be taken to be the classifying space of a certain monoidal category $\mathcal{C}_{1,3}^{Br}$ whose morphisms are braids in $\bR^3$. There exists a homotopical approximation $\mathcal{C}_{1,3}^{hBr}$ (whose morphisms are no longer braids, but something like ``virtual braids") receiving a functor from $\mathcal{C}_{1,3}^{Br}$. The classifying space of the homotopical approximation may be computed by the methods of this paper to be $\Omega^2 S^2$, which is well-known to be the group completion of $\mathcal{B}$. This suggests that it may be useful to study spaces of manifolds with rigid structure by studying their analogues with homotopical structure, which the methods of this paper allow one to compute.

\subsection{Outline}
In \S\ref{sec:chapter2:categorydefinition} we give a precise definition of the topological categories $\mathcal{C}_\theta(M)$. In \S\ref{sec:chapter2:SpacesOfManifolds} we recall the topological sheaf $\Psi_\theta(-)$ assigning to an open subset of $\bR^n$ the space of $d$-dimensional $\theta$-submanifolds of it, as defined by Galatius and the author \cite{GR-W}, and explain how to extend it to the site of all $n$-manifolds with boundary. In \S\ref{sec:chapter2:PosetModels} we relate the classifying space of $\mathcal{C}_\theta(M)$ to the space $\Psi_\theta(M \times \bR)$ if $M$ is compact, and to a variant of it if $M$ is not compact. In \S\ref{sec:chapter2:Micro-flexibility} we prove that the sheaf $\Psi_\theta$ is micro-flexible (in the sense of Gromov \cite{Gromov}), so that it satisfies an $h$-principle on open manifolds. In \S \ref{sec:HPrincipleApplication} we explain how to apply Gromov's $h$-principle to establish Theorem \ref{MainTheorem}. In \S\ref{sec:HomotopyInvariance}--\ref{sec:BraidMonoid} we discuss some applications of this theorem, which have been outlined above.

%% file: chap2.tex
\section{The cobordism category \texorpdfstring{$\mathcal{C}_{\theta}(M^n)$}{of submanifolds of M}} \label{sec:chapter2:categorydefinition}
Let $M$ be a smooth $n$-dimensional manifold (second countable and Hausdorff), possibly with boundary, which we will call the \textit{background manifold}. We will write $\mathring{M}$ for $M \setminus \partial M$. Let $\theta : \X \to \Gr_d(\bR^\infty)$ be a Serre fibration, write $\gamma_d \to \Gr_d(\bR^\infty)$ for the universal bundle and $\theta^* \gamma_d \to \X$ for the bundle classified by $\theta$.

\begin{defn}
For $V \to X$ and $U \to Y$ two vector bundles, we write $\Bun(U, V)$ for the subspace of $\map(U, V)$ (with the compact-open topology) consisting of those maps which are linear isomorphisms on each fibre.
\end{defn}

We define the category $\mathcal{C}_{\theta}(M)$ by
\begin{itemize}
\item An object is a triple $(a, X, \ell)$ where $\{ a \} \times X \subseteq \{ a \} \times \mathring{M}$ is a compact $(d-1)$-dimensional submanifold and $\ell \in \Bun(\bR \oplus TX, \theta^*\gamma_d)$.

\item A morphism $(a, X_a, \ell_a) \lra (b, X_b, \ell_b)$ is a pair $(W, \ell)$, where $W \subseteq [a,b] \times \mathring{M}$ is a compact $d$-dimensional submanifold, such that for some $\epsilon > 0$ (the \textit{collar size})
\begin{enumerate}[(i)]
	\item $W \cap ([a, a+ \epsilon) \times M) = [a, a+ \epsilon) \times X_a$,
	\item $W \cap ((b- \epsilon, b] \times M) = (b-\epsilon, b] \times X_b$,
	\item $\partial W = W \cap (\{ a, b \} \times M)$,
\end{enumerate}
and $\ell \in \Bun(TW, \theta^*\gamma_d)$. There are canonical isomorphisms $TW|_{X_a}  \cong \bR \oplus TX_a$ and $TW|_{X_b} \cong \bR \oplus TX_b$, and we insist that $\ell|_{X_a} = \ell_a$ and $\ell|_{X_b} = \ell_b$ under these identifications.
\end{itemize}

\begin{rem}
We could define a slightly different category by choosing once and for all an embedding of $M$ into $\bR^\infty$, and take maps $\ell : N \to \X$ which cover the Gauss map $N \to \Gr_d(\bR^\infty)$ given by $TN \to TM \to T\bR^\infty$. Taking $M = \bR^\infty$ and the identity embedding, this is the category defined by \cite[\S 5]{GMTW}. However, as remarked in that paper, it produces a category homotopy equivalent to the one we define here.
\end{rem}

\subsection{Topologising the category}
In this section we follow \cite{GMTW} closely, only making those slight changes necessary to take the background manifold $M$ into account. We will topologise the objects by firstly considering a parametrised version and then dividing out by the action of reparametrisation. A parametrised object is a quadruple $(a, X, \ell, e)$ where $e : X \hookrightarrow \{a\} \times \mathring{M}$ is a smooth embedding, and $\ell: \bR \oplus TX \lra \theta^*\gamma_d$. Then we would topologise these objects as
$$\mathbb{R} \times \coprod_{[X]} \Emb(X, \mathring{M}) \times \Bun(\bR \oplus TX, \theta^* \gamma_d)$$
where the disjoint union of over all diffeomorphism types of compact $(n-1)$-manifolds $[X]$, and the embedding space has the $C^\infty$ topology. Passing to the unparameterised version by taking the quotient by the action of $\Diff(X)$ then gives a natural bijection between $ob(\mathcal{C}_d^{\theta}(M))$ and
$$\mathbb{R} \times \coprod_{[X]} \Emb(X, \mathring{M}) \times_{\Diff(X)} \Bun(\bR \oplus TX, \theta^* \gamma_d)$$
and we topologise $ob(\mathcal{C}_d^{\theta}(M))$ so as to make this bijection  a homeomorphism.

\vspace{2ex}

To topologise the space of morphisms, we will again consider a parametrised version. Let $(W, h_0, h_1)$ be an abstract collared cobordism from $M_0$ to $M_1$, so a compact $d$-manifold with disjoint embeddings
$$h_0 : [0, 1) \times M_0 \lra W$$
$$h_1 : (0, 1] \times M_1 \lra W$$
such that $\partial W = h_0(\{0\} \times M_0) \coprod h_1(\{1\} \times M_1)$. For each $0 < \epsilon < \frac{1}{2}$ let $\Emb^{\epsilon}_\partial(W, [0,1] \times \mathring{M})$ be the space of all embeddings $j$ such that
$$j \circ h_0(t_0, x_0) = (t_0, j_0(x_0))$$
$$j \circ h_1(t_1, x_1) = (t_1, j_1(x_1))$$
for all $t_0 \in [0, \epsilon)$, $t_1 \in (1-\epsilon, 1]$ and $x_i \in M_i$, for some embeddings $j_i : M_i \hookrightarrow M$. We then define
$$\Emb_\partial(W, [0,1] \times \mathring{M}) := \colim_{\epsilon \to 0} \Emb^{\epsilon}_\partial(W, [0,1] \times \mathring{M}).$$
Let $\Diff^{\epsilon}_\partial(W)$ be the topological group of those diffeomorphisms of $W$ that restrict to product diffeomorphisms on the $\epsilon$-collars of $W$, and define
$$\Diff_\partial(W) := \colim_{\epsilon \to 0} \Diff^{\epsilon}_\partial(W).$$
Then the group $\Diff_\partial(W)$ acts continuously on $\Emb_\partial(W, [0,1] \times \mathring{M})$.

A parametrised morphism is a quintuple $(a, b, W, \ell, e)$ where $a \leq b$, $e: W \hookrightarrow [a,b] \times \mathring{M}$ is an embedding with collar size $\epsilon$ (which defines embedded collars $h_0, h_1$ as above on $W$) and $\ell \in \Bun(TW, \theta^*\gamma_d)$. We topologise these parametrised morphisms as (here ``objects'' denotes the set of identity morphisms, which we topologise as we did the parametrised object space)
$$\{\text{objects}\} \coprod \left ( \mathbb{R}^2_+ \times \coprod_{[W]} \Emb_\partial(W, [0,1] \times \mathring{M}) \times \Bun(TW, \theta^* \gamma_d) \right)$$
where the disjoint union is over all diffeomorphism types of triples $(W, h_0, h_1)$. Passing to the unparameterised version gives a bijection between $mor(\mathcal{C}_d^{\theta}(M))$ and
$$ob(\mathcal{C}_d^{\theta}(M)) \coprod \left ( \mathbb{R}^2_+ \times \coprod_{[W]} \Emb_\partial(W, [0,1] \times \mathring{M}) \times_{\Diff_\partial(W)} \Bun(TW, \theta^* \gamma_d) \right )$$
and we topologise the space of morphisms to make this bijection a homeomorphism.

%% file: chap3.tex
\section{Recollections on \texorpdfstring{$\Psi_\theta(-)$}{spaces of manifolds}}\label{sec:chapter2:SpacesOfManifolds}

\begin{defn}
For a $n$-manifold $M$, possibly with boundary, define the set $\Psi_\theta(M)$ to be the set of subsets $X \subset \mathring{M}$ which are closed as subspaces of $M$, have the structure of $d$-dimensional smooth submanifolds, and are equipped with a $\theta$-structure $\ell : TX \to \theta^* \gamma_d$.
\end{defn}

Galatius and the author \cite{GR-W} have defined a topology on $\Psi_\theta(U)$ when $U$ is an open subset of Euclidean space $\bR^n$.

\begin{thm}[\cite{GR-W}]
$\Psi_\theta(-)$ defines a continuous functor and a sheaf of topological spaces on the site $\mathcal{O}(\bR^n)$ of open subsets of $\bR^n$.
\end{thm}

We explain here how to promote $\Psi_\theta(-)$ to a sheaf of topological spaces on the site of \textit{all} $n$-manifolds (without boundary!) and open embeddings. There is a unique way to do this, by the following general theorem whose proof we also include.

\begin{thm}\label{thm:SheafExtension}
Suppose given a commutative square
\begin{diagram}
\mathcal{O}(\bR^n) & \rTo^{\Psi^t_{\mathcal{O}(\bR^n)}} & \mathbf{Top} \\
\dTo & \ruTo[dotted]^{\Psi^t} & \dTo \\
\mathcal{O}_n & \rTo^{\Psi} & \mathbf{Set}
\end{diagram}
where $\Psi$ is a sheaf of sets on the site $\mathcal{O}_n$ of all smooth $n$-manifolds and open embeddings, and $\Psi^t_{\mathcal{O}(\bR^n)}$ is a continuous sheaf of topological spaces the site of open subsets of $\bR^n$. Then there exists a unique continuous sheaf of topological spaces $\Psi^t$ making the diagram commute.
\end{thm}

Such a sheaf $\Psi^t$, if it exists, should have the following property. Call an open set $U \subseteq M$ an \textit{affine patch} if it is diffeomorphic to an open subset of $\bR^n$. Let $\{  U_i \lra M \}_{i \in I}$ be an open cover of $M$ by affine patches. Then
\begin{diagram}
\Psi(M) & \rTo & \prod_{i \in I} \Psi(U_i) & \pile{\rTo \\ \rTo} & \prod_{(i, j) \in I^2} \Psi(U_i \cap U_j)
\end{diagram}
is an equaliser diagram in $\mathbf{Set}$. The $U_i$ have open embeddings into Euclidean space, so choosing one defines a topology on $\Psi(U_i) \cong \Psi^t_{\mathcal{O}(\bR^n)}(U_i)$, which is independent of the choice of embedding (as $\Psi^t_{\mathcal{O}(\bR^n)}(-)$ is a continuous sheaf). Similarly for the $U_i \cap U_j$. This ought to also be an equaliser diagram in $\mathbf{Top}$, so we may choose to topologise $\Psi(M)$ to make it so: that is, give $\Psi(M)$ the subspace topology.

\begin{lem}
This topology on $\Psi(M)$ is well-defined.
\end{lem}
\begin{proof}
Let $\{ U'_k \to M \}_{k \in K}$ be another open cover by affine patches, and write $\Psi^t(M)'$ for the topology defined using this cover. We may suppose without loss of generality that the cover $\{U'_k\}$ is a \textit{refinement} of the cover $\{ U_i \}$. Let $K_i \subset K$ be such that $\cup_{k \in K_i} U'_k = U_i$. There is then a map of equaliser diagrams
\begin{diagram}
\Psi^t(M) & \rTo & \prod_{i \in I} \Psi^t(U_i) & \pile{\rTo \\ \rTo} & \prod_{(i, j) \in I^2} \Psi^t(U_i \cap U_j) \\
\dTo[dotted] & & \dTo & & \dTo \\
\Psi^t(M)' & \rTo & \prod_{i \in I} \prod_{k \in K_i} \Psi^t(U'_k) & \pile{\rTo \\ \rTo} & \prod_{(i, j) \in I^2} \prod_{(k, \ell) \in K_i \times K_\ell} \Psi^t(U'_k \cap U'_\ell) 
\end{diagram}
which defines a continuous bijection $\Psi^t(M) \to \Psi^t(M)'$ with underlying map of sets the identity. Let $V$ be an open subset of $\Psi^t(M)$, then it is $(\prod_{i \in I} V_i) \cap \Psi^t(M)$, for $V_i$ open in $\Psi(U_i)$. Furthermore, $V_i$ is $(\prod_{k \in K_i} V'_{k}) \cap \Psi^t(U_i)$ for $V'_{k}$ open in $\Psi^t(U'_{k})$. Thus the image of $V$ in $\Psi^t(M)'$ is $(\prod_{i \in I} \prod_{k \in K_i} V'_k) \cap \Psi^t(M)'$ so is open. Thus the map is also open, so a homeomorphism.
\end{proof}

\begin{lem}
Let $N^n \subset M^n$. Then the map $\Emb(N, M) \times \Psi^t(M) \to \Psi^t(N)$ is continuous. Thus $\Psi^t$ is a continuous presheaf.
\end{lem}
\begin{proof}
Let $\{ V_i\}_{i \in I}$ be an open cover of $N$ by affine patches and $\{ U_i\}_{i \in I \coprod J}$ be an open cover of $M$ by affine patches such that under the standard embedding $V_i$ is compactly contained in $U_i$. There is a open neighbourhood $W \subset \Emb(N, M)$ of the standard embedding such that if $f \in W$ then $f(V_i) \subset U_i$. Thus there is a continuous map $W \to \prod_{i \in I} \Emb(V_i, U_i)$. These observations define a continuous map of equaliser diagrams
\begin{diagram}
W \times \Psi^t(M) & \rTo & W \times \prod_{i \in I \coprod J} \Psi^t(U_i) & \pile{\rTo \\ \rTo} & W \times \prod_{(i, j) \in (I \coprod J)^2} \Psi^t(U_i \cap U_j) \\
\dTo[dotted] & & \dTo & & \dTo \\
\Psi^t(N) & \rTo & \prod_{i \in I} \Psi^t(V_i) & \pile{\rTo \\ \rTo} & \prod_{(i, j) \in I^2} \Psi^t(V_i \cap V_j) 
\end{diagram}
and so a continuous map of equalisers. This proves that the map $\Emb(N, M) \times \Psi^t(M) \to \Psi^t(N)$ is continuous on a neighbourhood of any fixed embedding, and so continuous everywhere.
\end{proof}

\begin{lem}
Let $\{ N_i \to M \}_{i \in I}$ be an open cover of $M$. Then
\begin{diagram}
\Psi^t(M) & \rTo & \prod_{i \in I} \Psi^t(N_i) & \pile{\rTo \\ \rTo} & \prod_{(i, j) \in I^2} \Psi^t(N_i \cap N_j)
\end{diagram}
is an equaliser diagram in $\mathbf{Top}$. Thus $\Psi^t$ is a sheaf of topological spaces.
\end{lem}
\begin{proof}
Let $\{ U_j \lra N_i\}_{j \in J_i}$ be a collection of open covers of the $N_i$ by affine patches. This gives a map of equaliser diagrams so a continuous map
\begin{diagram}
\Psi^t(M) & \rTo & eq & \rTo & \prod_{i \in I} \Psi^t(N_i) & \pile{\rTo \\ \rTo} & \prod_{(i, j) \in I^2} \Psi^t(N_i \cap N_j) \\
& & \dTo[dotted] & & \dTo & & \dTo\\
& & \Psi^t(M) & \rTo & \prod_{i \in I} \prod_{\ell \in J_i} \Psi^t(U_\ell) & \pile{\rTo \\ \rTo} & \prod_{(i, j) \in I^2} \prod_{(\ell, k) \in J_i \times J_j} \Psi^t(U_\ell \cap U_k)
\end{diagram}
which exhibits $\Psi^t(M)$ as the equaliser.
\end{proof}

Let us return to the special case $\Psi = \Psi_\theta$. Note that if $M$ has boundary then the restriction map $\Psi_\theta(M) \to \Psi_\theta(\mathring{M})$ is injective, and so we may topologise $\Psi_\theta(M)$ as a subspace of $\Psi_\theta(\mathring{M})$. If $N \subset M$ is an open submanifold there is a commutative diagram of sets
\begin{diagram}
\Psi_\theta(M) & \rTo & \Psi_\theta(N) \\
\dTo & & \dTo \\
\Psi_\theta(\mathring{M}) & \rTo & \Psi_\theta(\mathring{N}) 
\end{diagram}
where the lower map is continuous and the vertical maps are inclusions of subspaces. It follows that the upper map is also continuous, so restriction maps are continuous for manifolds with boundary. Similar one may show that $\Emb(N, M) \times \Psi_\theta(M) \to \Psi_\theta(N)$ is continuous for manifolds with boundary.

\begin{thm}
The functor $\Psi_\theta(-)$ is a continuous sheaf of topological spaces on the site of $n$-manifolds with boundary and open embeddings.
\end{thm}
\begin{proof}
Theorem \ref{thm:SheafExtension} establishes that it is a continuous sheaf on the sub-site of $n$-manifolds \textit{without} boundary, and the observation above shows it is a continuous presheaf on the site of $n$-manifolds with boundary. It remains to establish the sheaf property for covers of $n$-manifolds with boundary. Let $\{ N_i \to M \}_{i \in I}$ be a cover in the site of $n$-manifolds with boundary, so there is a commutative diagram of topological spaces
\begin{diagram}
\Psi_\theta(M) & \rTo^\iota & \prod_{i \in I} \Psi_\theta(N_i) & \pile{\rTo \\ \rTo} & \prod_{(i, j) \in I^2} \Psi_\theta(N_i \cap N_j) \\
\dTo & & \dTo & & \dTo \\
\Psi_\theta(\mathring{M}) & \rTo^{\mathring{\iota}} & \prod_{i \in I} \Psi_\theta(\mathring{N}_i) & \pile{\rTo \\ \rTo} & \prod_{(i, j) \in I^2} \Psi_\theta(\mathring{N}_i \cap \mathring{N}_j)
\end{diagram}
where the vertical maps are inclusions of subspaces, the lower row is an equaliser in $\mathbf{Top}$ and the top row is an equaliser in $\mathbf{Set}$. Certainly $\iota$ is a continuous map: we must show that $\iota$ is a homeomorphism onto its image. The map $\mathring{\iota}$ is a homeomorphism onto its image, and a continuous inverse of it from its image carries $\mathrm{Im}(\iota)$ bijectively onto $\Psi_\theta(M)$, as required.
\end{proof}

\subsection{Covariance}\label{sec:Covariance}
If $N \hookrightarrow M$ is a closed embedding of manifolds with boundary, there is a covariant map
$$\Psi_\theta(N) \lra \Psi_\theta(M)$$
given by treating a smooth submanifold of $N$ (which is closed as a subspace) as a smooth submanifold of $M$ (closed as a subspace).

\subsection{Evaluating \texorpdfstring{$\Psi_\theta(\bR^n)$}{psi}}\label{sec:EvaluatingOnRn}
Recall \cite[Theorem 3.22]{GR-W} that there is a $GL_n(\bR)$-equivariant homotopy equivalence
$$\Psi_\theta(\bR^n) \simeq \Th(\gamma_d^\perp \to \Gr_d^\theta(\bR^n)).$$
The $GL_n(\bR)$-equivariance was not mentioned in that paper, but the proof given can be seen to be equivariant.

%% file: chap4.tex
\section{Poset models}\label{sec:chapter2:PosetModels}

\begin{defn}
Let $\Psi_{\theta}^{c}(\bR \times M)$ be the colimit of $\Psi_{\theta}(\bR \times K)$ over all $K \subset M$ compact submanifolds (with boundary). In particular, if $M$ is compact $\Psi_{\theta}^{c}(\bR \times M) = \Psi_{\theta}(\bR \times M)$.
\end{defn}

The aim of this section is to identify the homotopy type of the cobordism category $\mathcal{C}_\theta(M)$ in terms of the space $\Psi_\theta^c(\bR \times M)$.
\begin{thm}\label{CategoryWeakEq}
Let $M$ be a smooth manifold, possibly with boundary. There is a weak homotopy equivalence
$$B \mathcal{C}_\theta(M) \simeq \Psi_{\theta}^{c}(\bR \times M).$$
\end{thm}
\begin{cor}
For $M$ compact this reduces to a homotopy equivalence $B \mathcal{C}_\theta(M) \simeq \Psi_{\theta}(\bR \times M)$.
\end{cor}

To prove this theorem we introduce several intermediate spaces.

\begin{defn}\label{PosetModelsDefinition}
To ease notation, define the space
$$D := \Psi_{\theta}^{c}(\bR \times M)$$
which we will consider as a topological category with only identity morphisms. Also define
$$D^{\pitchfork} := \{ (a, (A, \ell)) \in \mathbb{R} \times D \, | \, A \pitchfork \{ a \} \times M \}$$
and note it is non-empty by Sard's theorem, and that it is a poset via
$$(a, A, \ell) \leq (b, B, L) \,\,\, \text{if and only if} \,\,\, (A, \ell) = (B, L) \,\,\, \text{and} \,\,\, a \leq b.$$
We topologise $D^\pitchfork$ as a subspace of $\mathbb{R}^\delta \times D$, where $\mathbb{R}^\delta$ is the real line with the discrete topology. We also define a subposet $D^\perp$ of $D^\pitchfork$ of \textit{locally tubular} submanifolds: those elements of $D^\pitchfork$ such that there exists some $\epsilon > 0$ such that
$$A \cap ((a - \epsilon, a + \epsilon) \times M) = (a - \epsilon, a + \epsilon) \times (A \cap \{ a \} \times M)$$
and topologise $D^\perp$ with the subspace topology.
\end{defn}

It is clear that there are continuous functors between these, when we consider the posets as topological categories:
\begin{enumerate}[(i)]
	\item $u : D^\pitchfork \lra D$ that forgets the real value,
	\item $i : D^\perp \lra D^\pitchfork$ that includes the locally tubular submanifolds into the transverse submanifolds,
	\item $c : D^\perp \lra \mathcal{C}_\theta (M)$ which on objects is
$$(a, A, \ell) \mapsto (a, \{ a \} \times M \cap A, \ell \circ I)$$
where $I : \mathbb{R} \oplus T(\{ a \} \times M \cap A) \lra TA$ is the bundle map induced by the restriction map of the tangent bundle of $A$ to the submanifold $\{ a \} \times M \cap A$, where the factor of $\mathbb{R}$ is mapped to $TA$ by $(t, a')$ being sent to the vector of length $t$ at $a' \in A$ in the positive real direction of $\mathbb{R} \times M$. This is tangent to $A$ as $A$ is locally tubular at $a$.

On morphisms the functor $c$ is given by
$$\left \{ (a, A, \ell) \to (b, A, \ell) \right \} \mapsto \left ( a, b, A|_{[a,b]}, \ell|_{[a,b]} \right )$$
Note that the locally tubular property of elements of $D^\perp$ ensures that $A|_{[a,b]}$ has an $\epsilon$-collar, so the map does give an element of the cobordism category.
\end{enumerate}

We claim that each of these functors induces a (weak) homotopy equivalence on the level of classifying spaces, which will immediately prove Theorem \ref{CategoryWeakEq} via the zig-zag of homotopy equivalences
$$\Psi_{\theta}^{c}(M \times \bR) = D \overset{Bu}\lla BD^\pitchfork \overset{Bi}\lla BD^\perp \overset{Bc}\lra B \mathcal{C}_d^\theta(M).$$
In fact more is true: $i$ and $c$ are \textit{homotopy equivalence of categories}, that is, they induce levelwise homotopy equivalences on simplicial nerves. For the first weak equivalence we require the following technical result, from \cite[Lemma 3.4]{galatius-2006}, which we reproduce here.

\begin{lem}\label{etale}
let $\mathcal{C}$ be a topological category and $X$ be a space considered as a topological category with only identity morphisms. Let $f: \mathcal{C} \to X$ be a functor that is \'{e}tale on each of the spaces of objects and morphisms. Suppose that for each $x \in X$ the space $Bf^{-1}(x)$ is contractible. Then $Bf : B\mathcal{C} \to X$ is a weak homotopy equivalence.
\end{lem}

\begin{prop}\label{ForgetfulEquivalence}
The map $Bu : BD^\pitchfork \to D$ is a weak homotopy equivalence.
\end{prop}
\begin{proof}
Note that $N_k u : N_k D^\pitchfork \lra D$ is \'{e}tale, as $\mathbb{R}$ has been given the discrete topology and if $A \pitchfork \{ a \} \times M$ then all $(A', \ell')$ in some open neighbourhood of $(A, \ell)$ are also transverse to $\{a\} \times M$. The preimage of $(A, \ell)$ is the subset of $ a \in \mathbb{R}^\delta$ such that $A \pitchfork \{ a \} \times M$, so is non-empty by Sard's theorem, discrete and totally-ordered. Hence its classifying space is contractible. We may then apply Lemma \ref{etale}, and the result follows.
\end{proof}

\begin{prop}
The functor $i : D^\perp \to D^\pitchfork$ is a homotopy equivalence of categories.
\end{prop}
\begin{proof}
This follows by the direct analogue of \cite[Lemma 3.4]{GR-W} to $\bR \times M$ instead of $\bR^n$, whose proof still goes through in this setting. This gives a canonical way of straightening submanifolds near to a regular value of the height function.
\end{proof}

\begin{prop}
The functor $c : D^\perp \to \mathcal{C}_\theta (M)$ is a homotopy equivalence of categories.
\end{prop}
\begin{proof}
We can define a map $r_k : N_k \mathcal{C}_\theta(M) \to N_k D^\perp$ that takes a $k$-fold morphism $(a_i, a_{i+1}, A_i, \ell_i)$ for $1 \leq i \leq k$ to the submanifold $A$ of $\mathbb{R} \times M$ given by $\bigcup A_i$, and extended to the left of $\{ a_1 \} \times M$ by $(-\infty, a_1] \times (\{ a_1 \} \times M \cap A_1)$, and to the right of $\{ a_{k+1} \} \times M$ by $[a_{k+1}, \infty) \times (\{ a_{k+1} \} \times M \cap A_k)$. As these are composable morphisms in the category, the $\ell_i$ agree where both are defined and so we can define a bundle map $\ell$ over the whole submanifold, extending to the left and right constantly. As morphisms in $\mathcal{C}_\theta(M)$ are collared, $A$ is locally tubular at the points $a_i$, and so $(a_1,..., a_{k+1}, A, \ell)$ defines an element in $N_kD^\perp$.

It is clear that $N_k c \circ r_k = \mathrm{Id}$, but the reverse composition is not the identity. Instead, it takes (forgetting about $\ell$ for a moment) a submanifold $A$ and $(k+1)$ points $a_i$ where it is transverse, and returns a manifold $A'$ that is identical to A in between $a_1$ and $a_{k+1}$ but extended constantly outside of it. This map is, however, homotopic to the identity.

\begin{figure}[h]
\centering
\includegraphics[bb = 114 614 326 773]{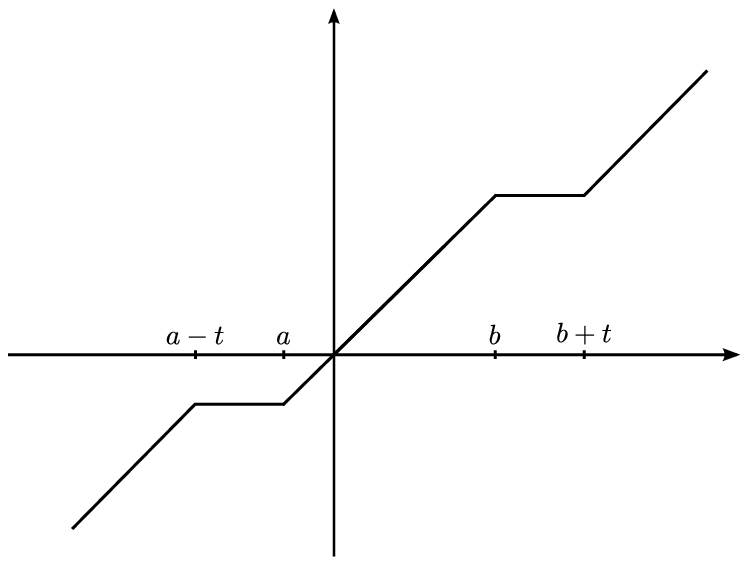}
\end{figure}

Let us write $x_1 : \bR \times M \to \bR$ for the projection. Let $\varphi_t^{a, b} : \mathbb{R} \times M \lra \mathbb{R} \times M$ be the function as in the diagram on the first factor and the identity on the second. Let $(a_1, ..., a_{k+1}, A, \ell) \in D^\perp$ and note that for each $t$, even though $\varphi_t^{a_1,a_{k+1}}$ is neither smooth nor invertible, $A_t := (\varphi_t^{a_1,a_{k+1}})^{-1}(A)$ is again a smooth manifold as near $x_1^{-1}(a_1)$ and $x_1^{-1}(a_{k+1})$ the manifold is locally tubular. It is identical to $A$ inside $x_1^{-1}[a_1, a_{k+1}]$, agrees with translated copies of $x_1^{-1}((-\infty, a_1) \coprod (a_{k+1}, \infty)) \cap A$ inside of $x_1^{-1}((-\infty, a_1-t) \coprod (a_{k+1}+t, \infty))$, and is  tubular in $x_1^{-1}((a_1-t, a_1) \coprod (a_{k+1},a_{k+1}+t))$.

We may define a $\theta$-structure $\ell_t : TA_t \to \theta^*\gamma_d$ to be that of $A$ inside $x_1^{-1}((-\infty, a_1-t) \coprod (a_1, a_{k+1}) \coprod (a_{k+1}+t, \infty))$, where there is a canonical embedding of this part of $A_t$ into $A$. On $x_1^{-1}[a_1-t, a_1]$ (resp. $x_1^{-1}[a_{k+1}, a_{k+1}+t]$) the manifold $A_t$ is tubular and so agrees with $\bR \times (x_1^{-1}(a_1) \cap A)$ (resp. $\bR \times (x_1^{-1}(a_{k+1}) \cap A)$), which has a canonical $\theta$-structure induced by that of $A$. We may use this to put a $\theta$-structure on $A_t$ in $x_1^{-1}[a_1-t, a_1]$ (resp. $x_1^{-1}[a_{k+1}, a_{k+1}+t]$) that is compatible with the one already defined on the remainder.

We may define a homotopy
\begin{eqnarray*}
H :  [0, \infty] \times N_k D^\perp & \lra & N_k D^\perp \\ 
(t, a_1,...,a_{k+1}, A, \ell) & \mapsto & (a_1,...,a_{k+1}, A_t, \ell_t) \quad \text{if} \,\,\, t < \infty \\
(\infty, a_1,...,a_{k+1}, A, \ell) & \mapsto & r_k \circ N_kc(a_1,...,a_{k+1}, A, \ell)
\end{eqnarray*}
which can easily be seen to be continuous, as it is eventually constant inside any compact subset of $\bR \times M$. It is a homotopy from the identity map on $D^\perp$ to $r_k \circ N_kc$, as required.
\end{proof}

%% file: chap5.tex
\section{Micro-flexibility of the sheaf \texorpdfstring{$\Psi_\theta(-)$}{of submanifolds}}\label{sec:chapter2:Micro-flexibility}

Let $\Psi : \mathcal{O}(M) \to \mathbf{Top}$ be a sheaf of topological spaces defined on a manifold $M$. If $C \subset M$ is a closed subset, one usually defines the value of the sheaf $\Psi$ on $C$ as the colimit of its values on open neighbourhoods of $C$, $\Psi(C) := \colim_{U \supset C} \Psi(U)$. However if the sheaf is one of topological spaces, the space so obtained tends to be pathological. One can develop the theory of \textit{quasi-topological spaces} (a generalisation of the category $\mathbf{Top}$ having better behaved colimits) to circumvent these problems, but for our purpose it is simpler to spell things out. The approach to this theory via quasi-topological spaces has been worked out in detail by Ayala \cite{Ayala}.

\begin{defn}
A sheaf $\Psi$ of topological spaces on a manifold $M$ is said to be \textit{micro-flexible} if for every inclusion $C' \subset C$ of compact subspaces of $M$ and each pair $U' \subset U$ of open sets containing $C'$, $C$ respectively, every lifting problem
\begin{diagram}
& & \{0\} \times P & \rTo & \Psi(U \supset C) \\
& \ldTo & \dTo & \ruTo(4,2)[dotted] & \dTo \\
[0, \epsilon] \times P & \rTo &[0,1] \times P & \rTo & \Psi(U' \supset C')
\end{diagram}
with $P$ a compact polyhedron has a solution for some $\epsilon > 0$, after perhaps passing to a smaller pair $\tilde{U}' \subset \tilde{U}$ still containing $C$ and $C'$ respectively.
\end{defn}

\begin{rem}
Once the theory has been developed, this is no more than saying that for inclusions of compact sets the restriction maps $\Psi(C) \to \Psi(C')$ are Serre micro-fibrations of quasi-topological spaces.
\end{rem}

In order to show $\Psi_\theta$ is micro-flexible, we first recall some of the properties of the spaces $\Psi_\theta(U)$ proved in \cite{GR-W}, in particular the smooth structure on these spaces.

\begin{defn}
Let $X$ be a smooth manifold, $f : X \to \Psi_\theta(U)$ a continuous map, and write $f(x)=(M_x, \ell_x)$. There is a subset $\Gamma_f \subset X \times U$ called the \textit{graph of $f$} defined by $\cup_{x \in X} \{x\} \times M_x \subset X \times U$. It has a \textit{vertical tangent bundle} $T^v \Gamma_f$ defined by $\cup_{x \in X} \{x\} \times TM_x \subset X \times TU$. The maps $\ell_x$ determine a bundle map $T^v\Gamma_f \to \theta^*\gamma_2$.

Let us say \textit{$f$ is smooth} if $\Gamma_f$ is a smooth submanifold of $X \times U$ and the projection $\pi_1 : \Gamma_f \to X$ is a submersion. More generally, say \textit{$f$ is smooth near $(x,u)$} if there are neighbourhoods $A \ni x$ and $B \ni u$ such that
$$A \lra X \lra \Psi_\theta(U) \lra \Psi_\theta(B)$$
is smooth. Say \textit{$f$ is smooth near a closed set} $C \subset X \times U$ if it is smooth near each point of $C$.
\end{defn}

It is not difficult to show that the set of smooth maps $X^k \to \Psi_\theta(U^n)$ is in bijection with the set of submersions $\pi_1 : \Gamma^{k+d} \subset X \times U \to X$ with a $\theta$-structure on the vertical tangent bundle \cite[Lemma 2.16]{GR-W}.

Furthermore, given a continuous map $f : X \to \Psi_\theta(U)$, an open $V \subset X \times U$ and an open $W \subset X \times U$ such that $\bar{V} \subset \mathring{W}$, there is a homotopy starting at $f$ which is supported in $W$ and is smooth on $(0,1] \times V$. If $f$ is already smooth on an open set $A \subset X$ the homotopy may be assumed to be smooth on $[0,1] \times A$ \cite[Lemma 2.17]{GR-W}.

\vspace{2ex}

Let $\tau: X \times U \lra [ 0,1 ]$ be a smooth function, and $F_\tau : [ 0,1 ] \times X \times U \lra [ 0,1 ] \times X \times U$ be given by $(t, x, u) \mapsto (t \tau(x,u), x, u)$. If $f : [0,1] \times X \lra \Psi_\theta(U)$ is a homotopy, then $\Gamma_f \subseteq [0,1] \times X \times U$. The following lemma gives a criterion for the set $\Gamma := F_\tau^{-1}(\Gamma_f)$ to be the graph $\Gamma_{f_\tau}$ of some continuous map $f_\tau$.

\begin{lem}\label{Deformation}
With the notation above, if either of the following hold
\begin{enumerate}[(i)]
	\item $\tau(x, u)$ is independent of $u$, 
	\item $f$ is smooth and $F_\tau |_{\{ t \} \times X \times U}$ is transverse to $\Gamma_f$ for all $t$,
\end{enumerate}
then $\Gamma = \Gamma_{f_\tau}$ for a continuous map $f_\tau$.
\end{lem}
\begin{proof}
In case (i), let $\tau(x, u) = \sigma(x)$. Define $f_\tau$ as
$$[0,1] \times X \lra [0,1] \times X \overset{f} \lra \Psi_\theta(U)$$
where the first map is $(t,x) \mapsto (\sigma(x)t, x)$. Let $(t,x,u) \in \Gamma$, so $(t \sigma(x), x, u) \in \Gamma_f$, so $u \in f(t\sigma(x), x)$. Thus $u \in f_\tau(t, x)$, so $(t,x,u) \in \Gamma_{f_\tau}$. It is easy to see that the reverse inclusion also holds.

In case (ii), by the transversality hypothesis $F_\tau^{-1}(\Gamma_f) = \Gamma$ is a smooth manifold. $F_\tau : \Gamma_f \lra \Gamma$ is a diffeomorphism over $[0,1] \times X$ and so $\Gamma \lra [0,1] \times X$ is also a submersion. By \cite[Lemma 2.16]{GR-W} $\Gamma = \Gamma_{f_\tau}$ for some continuous map $f_\tau : [0,1] \times X \to \Psi_\theta(U)$.
\end{proof}

Note that it is only necessary that one of these is satisfied on some neighbourhood of each point, to prove continuity at that point. The following proposition (analogous to \cite[Proposition 4.14]{galatius-2006}) is the main technical tool required to show micro-flexibility. It says that given a compact $K \subset U$ and a homotopy $f_t : P \to \Psi_\theta(U)$, there is another homotopy $g_t$ defined for $t \in [0, \epsilon)$ starting from the same map, equal to $f_t$ near $K$ but constant outside of some other compact $C \subset U$.

\begin{prop}\label{micro}
Let $K \subseteq U$ be compact and $P$ be a polyhedron. Let $f : P \times [0,1] \lra \Psi_\theta(U)$ be a continuous map. Then there exists an $\epsilon > 0$ and a continuous map $g: P \times [0, \epsilon] \lra \Psi_\theta(U)$ such that
\begin{enumerate}[(i)]
	\item $f|_{P \times [0, \epsilon]}$ agrees with $g$ on a neighbourhood of $K$,
	\item $g|_{P \times \{ 0 \}}$ agrees with $f|_{P \times \{ 0 \}}$,
	\item There is a compact set $C \subseteq U$ such that
$$P \times [0, \epsilon] \overset{g} \lra \Psi_\theta(U) \lra \Psi_\theta(U \setminus C)$$
factors through the projection $\pi_1 : P \times [0, \epsilon] \lra P$, i.e.\ is constant along $[0, \epsilon]$
\end{enumerate}
\end{prop}
\begin{proof}
Choose a smooth function $\tau : P \times U \lra [0,1]$ that is constantly 1 on a neighbourhood of $P \times K$, and with compact support. Let $A \subseteq U \setminus K$ be a closed set such that $\tau$ is locally constant on $P \times (U \setminus A)$. Let $B \subseteq U \setminus A$ be a closed neighbourhood of $K$. By \cite[Lemma 2.17]{GR-W}, we may suppose that $f$ is smooth near $P \times A$, and unchanged near $P \times B$.

Let $F_\tau : [0,1] \times P \times U \lra [0,1] \times P \times U$ be given by $(t, p, u) \mapsto (t \tau(p, u), p, u)$. The transversality condition of Lemma \ref{Deformation} is satisfied on $P \times A$ at $t=0$, and so it is for $t \in [0, \epsilon]$ for some $\epsilon>0$, as transversality is an open condition. On $P \times (U \setminus A)$, $\tau$ is locally constant and so the first condition of Lemma \ref{Deformation} is satisfied. Thus $F_\tau^{-1}(\Gamma_f) \cap ([0, \epsilon] \times P \times U) = \Gamma_{f_\tau}$ for a (unique) continuous $f_\tau$. Choose $g := f_\tau : [0, \epsilon] \times P \lra \Psi_\theta(U)$.

Part (i) is satisfied as $\tau(u, p) = 1$ on a neighbourhood of $P \times K$ and so $F_\tau$ is the identity on a neighbourhood of $[0,1] \times P \times K$, so $f|_{[0, \epsilon] \times P}$ and $f_\tau$ agree on this neighbourhood of $K$.

Part (ii) is satisfied as $F_\tau (0,p, x) = (0,p, x)$ so $f$ and $f_\tau$ agree at 0.

Part (iii) is satisfied by taking $C$ to be such that $\mathrm{supp}(\tau) \subseteq P \times C$. For $x$ outside of $C$, $F_\tau(t,p, x) = (0,p, x)$ and so $\Gamma_{f_\tau} \cap [0,\epsilon] \times P \times (U \setminus C) = [0, \epsilon] \times \Gamma_{f_\tau(0, p, x)}$. Thus the composite factors through the projection as required.
\end{proof}

Finally, we will use this proposition to prove micro-flexibility of the sheaf $\Psi_\theta(-)$.

\begin{thm}\label{thm:Microflexible}
The sheaf $\Psi_\theta(-)$ is micro-flexible.
\end{thm}
\begin{proof}
Let $K' \subseteq K \subseteq V$ be compact subsets. To show micro-flexibility it is enough to show that for all open sets $K' \subset U'$, $K \subset U$ such that $U' \subseteq U$, and all diagrams
\begin{diagram}
\{ 0 \} \times P & \rTo^{h_0} & \Psi_\theta(U) \\
\dTo & & \dTo \\
[0,1] \times P & \rTo^h & \Psi_\theta(U')
\end{diagram}
where $P$ is a compact polyhedron, there is an $\epsilon > 0$ and a $l : [0, \epsilon] \times P \lra \Psi_\theta(U)$ extending $h_0$ over $h|_{[0, \epsilon] \times P}$ (possibly after shrinking $U$ and $U'$).

We have $K' \subseteq U'$ a compact subset and $h : [0,1] \times P \lra \Psi_\theta(U')$ continuous. By Proposition \ref{micro} there is an $\epsilon > 0$ and a $g: [0,\epsilon] \times P \lra \Psi_\theta(U')$ that agrees with $h|_{P \times [0, \epsilon]}$ near $K$, agrees with $h|_{P \times \{0\}}$ on $P$, and is constant outside of a compact set $C \subset U'$. The following diagram of solid arrows then commutes
\begin{diagram}
\{ 0 \} \times P & \rTo^{h_0} & \Psi_\theta(U) & \rTo & \Psi_\theta(U \setminus C) \\
\uTo^{\pi_2} & & \dTo[dotted] & & \dTo \\
[0, \epsilon] \times P & \rTo^g & \Psi_\theta(U') & \rTo & \Psi_\theta(U' \setminus C).
\end{diagram}
Thus thinking of the rightmost square as being cartesian, we get a continuous map from $[0, \epsilon] \times P$ to $\Psi_\theta(U)$, which restricts on $\{0\} \times P$ to $h_0$. This is $l: [0, \epsilon] \times P \lra \Psi_\theta(U)$ and its restriction to $U'$ is $g$ by construction.

As $g$ agrees with $h|_{[0, \epsilon] \times P}$ near $K'$, we can pass to a smaller $K' \subset \tilde{U}' \subset U'$ on which they agree. Then $l$ is a lift extending $h_0$ and covering $[0,\epsilon] \times P \lra \Psi_\theta(U') \lra \Psi_\theta(\tilde{U}')$, as required.
\end{proof}

%% file: chap6.tex
\section{Applying Gromov's \texorpdfstring{$h$}{h}-principle}\label{sec:HPrincipleApplication}

We will apply the $h$-principle of Gromov \cite[p. 79]{Gromov}, which states that a micro-flexible, $\Diff(V)$-invariant sheaf $\Phi$ on an open manifold $V$ (possibly with boundary $\partial V$) satisfies the parametric $h$-principle. Let us explain this statement.

\vspace{2ex}

Suppose that $V$ is paracompact, so admits a Riemannian metric $d$ and hence an exponential map $\mathrm{exp} : TV \dashrightarrow V$ defined on a neighbourhood of the zero section. Let $\epsilon > 0$ be small enough that if we write $V_\epsilon = \{ v \in V \,\,|\,\, d(v, \partial V) \geq \epsilon \}$ then $V \setminus V_{2\epsilon}$ is a collar neighbourhood of $\partial V$ diffeomorphic to $\partial V \times [0,2\epsilon)$. There is a smooth function $r: V \to (0,\epsilon)$ such that at $v \in V$ the injectivity radius of the exponential map is at least $r(v)$ (this follows by the paracompactness of $V$ and the local existence of such a function). Let $D_r(TV)$ be the smooth open disc bundle where the fibre over $v$ is the open disc of radius $r(v)$ in $T_vV$, so that the exponential map defines a smooth map $\mathrm{exp}(r) : D_r(TV) \to V$. Note there is also a fibrewise diffeomorphism $D_r(TV) \cong TV$ that is the identity near the zero section, which with $\mathrm{exp}(r)$ gives a map
$$\rho : TV \lra V$$
with derivative the identity near the zero section.

Let us write $\Phi^{fib}(TV)$ for $F(V) \times_{GL_n(\bR)} \Phi(\bR^n)$ where $F(V)$ is the frame bundle of $V$. For an element $x \in \Phi(V)$ we may consider the map $V \to \Phi^{fib}(TV)$ over $V$ given by 
$$v \mapsto \{\rho|_{T_vV}\}^{-1}(x) \in \Phi(T_vV).$$
This defines a continuous map, the \textit{scanning map},
$$\Phi(V) \lra \Gamma(\Phi^{fib}(TV) \to V),$$
and to say $\Phi$ satisfies the parametric $h$-principle is to say this map is a weak homotopy equivalence. By Theorem \ref{thm:Microflexible} the sheaf $\Psi_\theta(-)$ is micro-flexible, and it is also equivariant.

\begin{cor}
Let $V$ be a smooth manifold without boundary which is open (i.e.\ has no compact components). Then the scanning map
$$\Psi_\theta(V) \lra \Gamma(\Psi_\theta^{fib}(TV) \to V)$$
is a weak homotopy equivalence.
\end{cor}

This corollary precisely says that the functor $\Psi_\theta(-)$ is \textit{linear}, in the sense of the manifold calculus of Weiss and Goodwillie \cite{WeissEmb, GoodwillieWeiss}. Note that by taking $V = \bR \times M$ for $M$ a compact manifold without boundary, we obtain Theorem \ref{MainTheorem} in the case of $M$ compact without boundary. The case of $M$ non-compact or with boundary is slightly more complicated.

\vspace{2ex}

There is a relative version of this construction for manifolds $V$ with boundary. Suppose $\Phi(-)$ takes values in \textit{pointed} topological spaces, and let $\Phi(V, \partial V) \subset \Phi(\mathring{V})$ be the subspace of those elements which restrict to $* \in \Phi(U \cap \mathring{V})$ for some open neighbourhood $U$ of $\partial V$. Then $\Phi(V, \partial V)$ has covariant extension maps: if $V \hookrightarrow U$ is a closed inclusion of a submanifold with boundary, there is a map $\Phi(V, \partial V) \to \Phi(U, \partial U)$ which extends sections to be $*$ outside of $V$. There is a scanning map
$$\Phi(V_{\epsilon}, \partial V_{\epsilon}) \lra \Gamma(\Phi^{fib}(T\mathring{V}) \to \mathring{V};\partial V)$$
to the space of sections which restrict to $*$ near $\partial V$. If $\Phi$ satisfies the parametric $h$-principle (i.e.\ $\Phi$ is micro-flexible and equivariant, and the pair $(V, \partial V)$ is open) then this map is a weak homotopy equivalence. Note that the source and target of this map are respectively homotopy equivalent to $\Phi(V, \partial V)$ and the space $\Gamma(\Phi^{fib}(TV) \to V;\partial V)$ of sections which take the value $*$ on $\partial V$.

\vspace{2ex}

By the remarks in \S\ref{sec:EvaluatingOnRn}, there is an $GL_n(\bR)$-equivariant map
$$T_\theta(\bR^n) \lra \Psi_\theta(\bR \oplus \bR^n)$$
which is also a weak homotopy equivalence. By Theorem \ref{thm:Microflexible} the sheaf $\Psi_\theta$ is micro-flexible, it is also equivariant and the pair $(\bR \times M, \bR \times \partial M)$ is open for any manifold $M$ so we have a zig-zag of weak homotopy equivalences
$$\Psi_\theta(\bR \times M) \simeq \Gamma(\Psi^{fib}_{\theta}(\bR \oplus TM) \to \bR \times M;\bR \times \partial M) \simeq \Gamma(T_\theta^{fib}(TM) \to M;\partial M).$$
This establishes Theorem \ref{MainTheorem} for all compact manifolds $M$, and the case of non-compact manifolds follows by the following general argument.

\begin{cor}\label{cor:ScanningMapEq}
Let $M$ be a smooth manifold, possibly with boundary. Let $\Psi_\theta^{c}(\bR \times M) \subseteq \Psi_{\theta}(\bR \times M)$ be the subspace of those $d$-submanifolds which are contained in $\bR \times K$ for some compact $K \subseteq M$. Then the scanning map
$$\Psi_\theta^{c}(\bR \times M) \lra \Gamma_{cM}(\Psi_\theta^{fib}(\bR \oplus TM) \to \bR \times M)$$
is a weak homotopy equivalence, where $\Gamma_{cM}$ denotes sections that are ``uniformly compactly supported in the $M$ direction'': supported in some $\bR \times K$ for $K \subset M$ compact.
\end{cor}
\begin{proof}
Let $\mathcal{K}$ be the category having as objects pairs $(K, w(K))$ of a compact submanifold $K$ of $M$ with piecewise smooth boundary, and a real number $w(K) \in (0, \inf\{r(k) \,|\, k \in K\})$ such that the inclusion $K_{w(K)} \hookrightarrow K$ is a homotopy equivalence. A morphism $(K, w(K)) \to (K', w(K'))$ in $\mathcal{K}$ exists if $K \subseteq K'$ and $K_{w(K)} \subseteq K'_{w(K')}$, in which case it is unique. For each $(K, w(K)) \in \mathcal{K}$, let $\Gamma(\Psi_{\theta}^{fib}(\bR \oplus TK) \to \bR \times K ; \bR \times \partial K)$ denote the space of sections which over $\bR \times \partial K$ take value $\emptyset \in \Psi_\theta(\bR \oplus T_kK)$.

The scanning map on $\bR \times M$ restricts to a map
$$\Psi_\theta(\bR \times K_{w(K)}) \lra \Gamma(\Psi_{\theta}^{fib}(D_{r}(\bR \oplus TK)) \to \bR \times K ; \bR \times \partial K)$$
and $K_{w(K)} \overset{\simeq} \hookrightarrow K$, so by an application of the discussion above this map is a weak homotopy equivalence.

As discussed in \S\ref{sec:Covariance}, given an inclusion $K \subseteq K'$ such that $K_{w(K)} \subseteq K'_{w(K')}$ there are extension maps
$$\Psi_\theta(\bR \times K_{w(K)}) \lra \Psi_\theta(\bR \times K'_{w(K')})$$
which come from including a closed submanifold of $\bR \times K_{w(K)}$ that is contained in the interior as a closed submanifold of $\bR \times K'_{w(K')}$ that is contained in the interior. These extension maps give a diagram of topological spaces indexed by $\mathcal{K}$,
$$(K, w(K)) \mapsto \Psi_\theta(\bR \times K_{w(K)})$$
and $\colim_\mathcal{K} \Psi_\theta(\bR \times K_{w(K)}) = \Psi_\theta^c(\bR \times M)$.

There are also extension maps
$$\Gamma(\Psi_{\theta}^{fib}(\bR \oplus TK) \to \bR \times K; \bR \times \partial K) \lra \Gamma(\Psi_{\theta}^{fib}(\bR \oplus TK') \to \bR \times K'; \bR \times \partial K')$$
given by extending sections that are constantly $\emptyset$ on the boundary of $K$ by $\emptyset$. These also give a diagram of topological spaces indexed by $\mathcal{K}$ with colimit $\Gamma_{cM}(\Psi_{\theta}^{fib}(TM \times \mathbb{R}) \to M \times \mathbb{R})$. The scanning construction gives a map of $\mathcal{K}$-diagrams which is levelwise a weak homotopy equivalence, and $\mathcal{K}$ is right filtered so colimits and homotopy colimits are weakly homotopy equivalent \cite[p. 331]{Bousfield-Kan}. The result follows.
\end{proof}

\begin{rem}
We did not use any properties of the sheaf $\Psi_\theta$ in this proof, and the same holds for any sheaf of pointed spaces $\Phi$, once it satisfies the parametric $h$-principle for compact manifolds with boundary.
\end{rem}

Theorem \ref{MainTheorem} now follows directly from Theorem \ref{CategoryWeakEq} and Corollary \ref{cor:ScanningMapEq}.

\begin{rem}
There is another approach to prove Corollary \ref{cor:ScanningMapEq}, analogous to that used by McDuff \cite{McDuff} for configuration spaces. It is clear that the scanning map
$$\Psi_\theta(D^n) \lra \Gamma (\Psi_\theta^{fib}(TD^n) \to D^n)$$
is a homotopy equivalence, as the right hand side is a space of sections over a contractible space, so is homotopy equivalent to $\Psi_\theta(\bR^n)$. It is also clear that the functor on the right hand side is \textit{homotopy excisive}, in the sense that it sends cocartesian squares of manifolds to homotopy cartesian squares of spaces.

To show the two functors agree on all manifolds, it is enough to show that $\Psi_\theta(-)$ is also homotopy excisive. However, to show they agree on open manifolds it is enough to show $\Psi_\theta(-)$ sends only \textit{certain} pushouts of manifolds to homotopy cartesian squares. As it certainly produces cartesian squares, if we can show restriction maps $\Psi_\theta(U) \to \Psi_\theta(V)$ are quasi-fibrations for certain pairs $V \subset U$ (as in \cite[p. 97]{McDuff}) the result follows.

This can be done: the necessary restriction maps \textit{are} quasi-fibrations, however the technical details are overly complicated, and we do not include them here.
\end{rem}

%% file: chap7.tex
\section{Homotopy invariance}\label{sec:HomotopyInvariance}
From the weak equivalence
$$B \mathcal{C}_\theta (M) \simeq \Gamma_c(T_\theta^{fib}(TM) \to M)$$
of Theorem \ref{MainTheorem}, it is clear that the homotopy type of $B \mathcal{C}_\theta (M)$ depends only on the proper homotopy type of the manifold $M$ and the isomorphism class of the once-stabilised tangent bundle (recall that $T_\theta(V)$ is really a functor of the once stabilised vector space $\bR \oplus V$).

\begin{thm}\label{NormalHomotopyEquivalence}
The weak homotopy type of the space $B \mathcal{C}_\theta (M)$ is an invariant of ``normal proper homotopy equivalence'', that is, proper homotopy equivalences of manifolds $f: M \to N$ of the same dimension such that $f^* \nu_N \cong \nu_M$, where $\nu$ denotes the stable normal bundle.
\end{thm}

\begin{proof}
This result will follow immediately once we show that one can recover uniquely up to isomorphism the once-stabilised tangent bundle $\bR \oplus TM$ from the stable normal bundle $\nu_M$ (or, equivalently, the stable tangent bundle). This follows as $BO(n+1) \to BO$ has $(n+1)$-connected homotopy fibre.
\end{proof}

Let us write $\mathbf{Man}$ for the category of finite-dimensional submanifolds of $\bR^\infty$ and closed embeddings. One may ask how far the covariant functor
$$B\mathcal{C}_\theta(-) : \mathbf{Man} \to \mathbf{Top}$$
is from being homotopy invariant. More precisely, consider the category $B\mathcal{C}_\theta(-) \downarrow \mathrm{hoFun}(\mathbf{Man}, \mathbf{Top})$ where $\mathrm{hoFun}(\mathbf{Man}, \mathbf{Top})$ is the localisation of the category of functors obtained by inverting those natural transformations that induce homotopy equivalences on all values of the functors. Is there an initial object $B\mathcal{C}_\theta(-) \to F(-)$ in the subcategory of \textit{homotopy invariant} functors under $B\mathcal{C}_\theta(-)$?

There is \textit{a} homotopy invariant functor under $B\mathcal{C}_\theta(-)$, given by
$$B\mathcal{C}_\theta(M) \to B\mathcal{C}_{\theta \times M}(\bR^\infty) \simeq \Omega^{\infty-1} \MTtheta \wedge M_+$$
defined by the embedding of $M$ in $\bR^\infty$. Thus if $F$ exists there is a natural transformation $F(-) \to \Omega^{\infty-1} \MTtheta \wedge -_+$. We claim that this must be an equivalence in the category described above.

Certainly $M \hookrightarrow \bR^n \times M$ is a closed embedding and a homotopy equivalence, so we should have $F(-) \overset{\simeq} \to F(\bR^n \times -)$, and furthermore
$$F(-) \overset{\simeq}\lra \hocolim_{n \to \infty} F(\bR^n \times -).$$
As $B\mathcal{C}_\theta(\bR^n \times M) \simeq \Gamma_c(\Omega^nT_\theta^{fib}(TM \oplus \bR^n) \to M)$ this last functor receives a natural map from
$$\hocolim_{n \to \infty}\Gamma_c(\Omega^nT_\theta^{fib}(T- \oplus \bR^n) \to -).$$
This homotopy colimit is formed using the map $\Sigma^n T_\theta(V) \to T_\theta(V \oplus \bR^n)$ with adjoint $T_\theta(V) \to \Omega^n T_\theta(V \oplus \bR^n)$ which when applied fibrewise to the tangent bundle of $M$ realises the map $B\mathcal{C}_\theta(M) \to B\mathcal{C}_\theta(\bR^n \times M)$. The colimit over $n$ gives
$$T_\theta(V) \lra T_\theta^{stable}(V) := \hocolim_n \Omega^n T_\theta(V \oplus \bR^n) \simeq \Omega^{\infty-1}(\MTtheta \wedge V^+),$$
a natural transformation to a functor taking values in infinite loop spaces. Then there is a natural equivalence of functors
$$\Gamma_c(T_\theta^{stable, fib}(T-) \to -) \overset{\text{Poincar\'{e} duality}}\simeq \Omega^{\infty-1} \MTtheta \wedge -_+,$$
so there is also a natural transformation $\Omega^{\infty-1} \MTtheta \wedge -_+ \to F(-)$. Thus the best homotopy invariant approximation (from the right) to $B\mathcal{C}_\theta(-)$ is $\Omega^{\infty-1}\MTtheta \wedge-_+$, degree-shifted $\MTtheta$ bordism theory.

\vspace{2ex}

Note that both here and in the next section we use the approach to Poincar\'{e} duality via ``parametrised spectra", which is described in, for example, \cite[p. 25]{CK}.

\section{Poincar\'{e} duality}\label{sec:PoincareDuality}

For simplicity, let us take the tangential structure of an orientation, $\theta : \Gr_d^+(\bR^\infty) \to \Gr_d(\bR^\infty)$, and let $M$ be an $n$-dimensional manifold. Then
$$\pi_0(B\mathcal{C}_\theta(M)) = \{ \text{smooth oriented $(d-1)$-submanifolds of $M$} \}/\sim$$
where the equivalence relation is that of cobordism in $[0,1] \times M$. This is a ``homological" functor of $M$.

\begin{rem}
If one wants to discuss ``cobordism classes of submanifolds of $M$", we claim that this is the correct notion. If instead one asks for submanifolds of $M$ modulo cobordisms in $M$ there are several problems: it is not clear if isotopic embeddings are equivalent (as one may not be able to embed an isotopy), and it is not clear that the cobordism relation is transitive.
\end{rem}
Theorem \ref{MainTheorem} identifies this set with
$$\pi_0 \Gamma_c(T_\theta^{fib}(TM) \to M),$$
a set of isotopy classes of (twisted) functions on $M$. This is a ``cohomological" functor of $M$, which suggests an interpretation of Theorem \ref{MainTheorem} as a form of unstable Poincar\'{e} duality which refines Poincar\'{e} duality for the homology theory $\MTtheta$. There is a homotopy commutative square
\begin{diagram}
B\mathcal{C}_\theta(M) & \rTo & B\mathcal{C}_{\theta \times M}(\bR^\infty) \simeq  \Omega^{\infty-1}(\MTtheta \wedge M_+) \\
\dTo^\simeq_{\text{Theorem \ref{MainTheorem}}} & & \dTo^\simeq_{\text{Poincar\'{e} duality}} \\
\Gamma_c(T_\theta^{fib}(TM) \to M) & \rTo & \Gamma_c(\Omega^{\infty-1} \MTtheta \wedge_{fib} TM^+ \to M)
\end{diagram}
where the lower map is induced by the composition
$$T_\theta(V) \lra T_\theta^{stable}(V) \simeq \Omega^{\infty-1} \MTtheta \wedge V^+$$
applied fibrewise to the tangent bundle of $M$. Thus we can consider Theorem \ref{MainTheorem} as giving a refinement of Poincar\'{e} duality from $\MTtheta$ bordism theory to bordism of submanifolds, which is not an additive theory.

\begin{example}
Let $M$ be a compact 3-manifold, and so parallelisable. Let $d=3$, so that $\pi_0(B\mathcal{C}_3^+(M))$ is the set of cobordism classes of oriented surfaces in $M$. The above observations identify this set with the set of free homotopy classes of maps from $M$ to $\Th(\gamma_3^\perp \to Gr^+_3(\bR^4)) \simeq \Sigma S^3_+ \simeq S^1 \vee S^4$, which is in natural bijection with $H^1(M ; \bZ)$ as the inclusion of a fibre $K(\bZ,1) = S^1 \to \Th(\gamma_3^\perp \to Gr^+_3(\bR^4))$ is 3-connected. The bijection
$$\{\text{Cobordism classes of oriented surfaces in $M$}\} \lra H^1(M;\bZ)$$
simply sends a surface to the Poincar\'{e} dual of the 2-cycle it represents. This recovers the well-known fact that every class in the second homology of a 3-manifold is representable by an embedded submanifold. Furthermore, this submanifold is unique up to cobordism in $[0,1] \times M$.
\end{example}

%% file: chap8.tex
\section{Configuration spaces of signed points}\label{sec:ConfigurationsOfSignedPoints}

In \cite{McDuff} McDuff introduced the space $C^{\pm}(M)$ of signed configurations in a manifold $M$, topologised so that positive and negative particles can come together and annihilate. In this section we will explore the relation between
\begin{enumerate}[(i)]
	\item the cobordism category $\mathcal{C}_1^+(M)$ having objects configurations of oriented points in $M$ and morphisms oriented 1-submanifolds of $[a,b] \times M$,
	\item the category $\mathrm{Path}(C^\pm(M))$ of continuous paths in $C^\pm(M)$.
\end{enumerate}
We will also show how these ideas identify the homotopy type of the spaces $C^\pm(M)$, using Theorem \ref{MainTheorem}, and hence recover McDuff's Theorem 1.3.

\subsubsection{A comparison functor}

Let us write $\mathcal{C}_1^+(M)$ for the cobordism category with morphisms oriented 1-manifolds (corresponding to the tangential structure $\theta : \Gr_1^+(\bR^\infty) \to \Gr_1(\bR^\infty)$), and $\mathrm{Path}(C^{\pm}(M))$ for the topological category of continuous Moore paths in $C^\pm(M)$. It is intuitively clear from the construction of $C_1^+(M)$ that there ought to be a functor
$$F': \mathcal{C}_1^{+}(M) \lra \mathrm{Path}(C^{\pm}(M))$$
by considering an oriented 1-manifold in $\bR \times M$ to be the trace of a path of signed configurations in $M$. However, this does not quite exist: on a 1-submanifold $W \subset \bR \times M$ the height function $x_1 : W \to \bR$ can be highly degenerate.

Instead, let $\mathcal{C}_1^{+}(M)^i \subset \mathcal{C}_1^{+}(M)$ denote the subcategory having the same space of objects, and spaces of morphisms given by those 1-dimensional oriented submanifolds $A \subset [a,b] \times M$ on which the height function $x_1 : A \to [a,b] \subset \bR$ has \textit{isolated critical points}. Let $\Psi_\theta^c(\bR \times M)^i \subset \Psi_\theta^c(\bR \times M)$ denote the subspace of those oriented 1-submanifolds $W \subset \bR \times M$ such that the map $x_1 : W \to \bR$ given by projection has isolated critical points.

\begin{prop}
The space $\Psi_\theta^c(\bR \times M)^i$ has the same weak homotopy type as $\Psi_\theta^c(\bR \times M)$. Furthermore there is a natural weak homotopy equivalence $B\mathcal{C}_1^+(M)^i \simeq \Psi_\theta^c(\bR \times M)^i$ and hence $\mathcal{C}_1^+(M)^i \to \mathcal{C}_1^+(M)$ induces an equivalence on classifying spaces.
\end{prop}
\begin{proof}
The first part holds as we have excluded a subspace of infinite codimension. The methods of \S \ref{sec:chapter2:PosetModels} show the second part, from which the third follows immediately.
\end{proof}

Now we may construct a functor
$$F: \mathcal{C}_1^{+}(M)^i \lra \mathrm{Path}(C^{\pm}(M)).$$
On an object $(a, X, \ell)$ it is given as follows: $X \subseteq M$ is a collection of points in $M$, and $\ell : \epsilon^1 \lra \gamma_1^+$ is a bundle map from the trivial line bundle over $X$ to $\gamma_1^+ \lra \Gr_1^+(\bR^\infty)$. Over each point of $X$ there are two such maps, corresponding to the two trivialisations of the trivial line bundle. The line bundle over $X$ has a preferred trivialisation given by the cobordism direction, so we can attach a sign $\pm 1$ to each point of $X$ depending on whether or not $\ell$ is orientation-preserving on the line bundle over that point. This determines a signed configuration in $M$.

On a morphism $(W, \ell) : (a, X_a, \ell_a) \lra (b, X_b, \ell_b)$ we associate the obvious path of configurations given by all the slices $W \cap \{ t \} \times M$. This clearly gives a path of configurations, except possibly at critical values $t$. In these cases, three things can be happening, depending on the height function $t$ on $W$:
\begin{enumerate}
\item $t$ has a local maximum at that point: this corresponds to two points with opposite signs approaching, so the correct thing to do is remove the point from the configuration. This will make the path be continuous for the topology on $C^\pm(M)$.
\item $t$ has a local minimum at that point: this corresponds to the inverse of the above situation, so we also remove the point from the configuration.
\item $t$ has an inflexion at that point: in this case the sign either side of the point is the same, so for continuity we attach that sign at the point.
\end{enumerate}
These conventions ensure that the functor $F$ is well defined.

\subsubsection{Paths in $C^\pm(M)$ and cobordism of oriented points}

Now that we have the functor $F$, it is tempting to claim that it is in some sense an equivalence of topological categories. We discuss this below, and leave it to the reader to deduce why this cannot be true. 

One difficulty is that connected morphisms (and objects) of $\mathcal{C}_1^+(M)^i$ have orientations which are chosen from large spaces having two contractible components (the spaces $\Bun(TW, \gamma_1^+)$). However we can easily pass to a homotopy equivalent category $\mathcal{C}_1^{rigid}(M)^i$ where orientations are sections of the unit sphere bundle (for a metric inherited from $M$), and so each connected 1-manifold admits precisely two orientations. Thus this is not a real problem.

Taking the category of ``smooth paths" in $C^\pm(M)$ (for a correct notion of smoothness) and the category $\mathcal{C}_1^{rigid}(M)^i$, the functor analogous to $F$ can be made to induce a continuous bijection on spaces of objects and morphisms. It is not however a homeomorphism: the topology on the space of paths is far coarser than that on the cobordism category. For example, consider the morphism $\emptyset \to \emptyset$ in $\mathcal{C}_1^{rigid}(\bR)$ represented by a standard circle embedded in $[0,1] \times \bR$. Considered as a path in $C^\pm(\bR)$ this is homotopic to the constant path at $\emptyset \in C^{\pm}(\bR)$, as it is given by a path from $\emptyset$ to $\{+1,-1\} \subset \bR$ composed with its reverse path. However it is not in the path component of the identity morphism $\mathrm{Id}_\emptyset$ in $\mathcal{C}_1^{rigid}(\bR)$.

In the next section we show that $B\mathcal{C}_1^+(M)^i$ and $B\mathrm{Path}(C^\pm(M)) \simeq C^\pm(M)$ are nevertheless homotopy equivalent.

\subsubsection{Germs of oriented 1-manifolds and $C^\pm(M)$}

For each $\epsilon \in (0, \infty]$ there is a continuous restriction map
$$\Psi_\theta^c(\bR \times M)^i \lra \Psi_\theta^c((-\epsilon, \epsilon) \times M)^i,$$
and we may write $\mathcal{G} := \colim_{\epsilon \to 0} \Psi_\theta^c((-\epsilon, \epsilon) \times M)^i$ for the space of germs of oriented 1-submanifolds near $\{0\} \times M$.

\begin{lem}\label{lem:Germs}
The map $\Psi_\theta^c(\bR \times M)^i \to \mathcal{G}$ sending a submanifold to its germ is a weak homotopy equivalence.
\end{lem}
\begin{proof}
All of the maps $\Psi_\theta^c(\bR \times M)^i \to \Psi_\theta^c((-\epsilon, \epsilon) \times M)^i$ are homotopy equivalences (by the usual scanning argument), and the colimit defining $\mathcal{G}$ is right filtered so is weakly homotopy equivalent to its homotopy colimit \cite[p. 331]{Bousfield-Kan}.
\end{proof}

There is a continuous surjection $f : \mathcal{G} \to C^\pm(M)$ given by intersecting manifolds with $\{0\} \times M$ and discarding any points which are critical points of the height function.

\begin{lem}\label{lem:GermSpaceContractibleFibres}
The map $f$ has contractible fibres.
\end{lem}
\begin{proof}
The fibre $f^{-1}(C)$ over a configuration $C$ consists of germs of those oriented 1-submanifolds of $\bR \times M$ with isolated critical points which intersect $\{0\} \times M$ in the configuration $C$ plus perhaps some critical points. We may push the critical points right or left (if they are minima and maxima respectively) to give a deformation retraction of $f^{-1}(C)$ into its subspace of germs of those 1-submanifolds which intersect $\{0\} \times M$ transversely in the configuration $C$. This is contractible by the usual scanning argument.
\end{proof}

If the map $f$ were locally well-behaved, the contractibility of the fibres would imply that it is a homotopy equivalence. However, $f$ has no such good local behaviour as, for example, a neighbourhood of $\emptyset \in C^\pm(M)$ contains configurations of arbitrarily many points, whereas $\emptyset \in \mathcal{G}$ is an open point.

\vspace{2ex}

Recall McDuff's filtration $C^\pm_k(M) \subset C^\pm(M)$ of the configuration space of positive and negative particles \cite{McDuff}, being the subspace of those configurations representable by at most $k$ negative and $k$ positive particles.  Fix a Riemannian metric on $M$, and let $U_k \subset C^\pm_k(M)$ be the open subset consisting of those configurations which are either in $C_{k-1}^\pm(M)$, or such that there is a unique (geodesically) closest pair of particles with opposite signs.

\begin{lem}
The map
$$g:f^{-1}(C_k^\pm(M)) \setminus f^{-1}(C_{k-1}^\pm(M)) \to C^\pm_k(M) \setminus C^\pm_{k-1}(M)$$
induced by $f$ is a quasi-fibration with contractible fibres, so a weak homotopy equivalence.
\end{lem}
\begin{proof}
The fibres are contractible as proved above. In the space $C^\pm_k(M) \setminus C^\pm_{k-1}(M)$ positive and negative particles cannot annihilate or be created, so this is just a configuration space with points labelled by $\pm 1$. As such this space has an obvious smooth structure, as an open submanifold of $M^{2k}$.

Let $\{C\} \in C^\pm_k(M) \setminus C^\pm_{k-1}(M)$ have an open neighbourhood $V$ which smoothly deformation retracts to $\{C\}$, via a map $\varphi_t : V \to V$. Define a homotopy $\overline{\varphi}_t : g^{-1}(V) \to g^{-1}(V)$ by first removing all critical points that lie in $\{0\} \times M$ as in the proof of Lemma \ref{lem:GermSpaceContractibleFibres}, and then translating the intersection point of germs according to $\varphi_t$. This gives a deformation retraction of $g^{-1}(V)$ into $g^{-1}(C)$, so $g$ is a local quasi-fibration, and hence a quasi-fibration.
\end{proof}

Note that $U_k$ has a canonical deformation retraction $h_t$ onto $C_{k-1}^\pm(M)$ as follows. On $C_{k-1}^\pm(M)$ it is the identity; on $U_k \setminus C_{k-1}^\pm(M)$ for $t \in [0,\tfrac{1}{2}]$ do nothing, and for $t \in [\tfrac{1}{2},1]$ slide the unique closest pair of particles with opposite signs together (quadratically fast) along the geodesic joining them, so that they cancel in the middle. For each $C \in U_k \setminus C_{k-1}^\pm(M)$ this homotopy gives a path which has a graph $\Gamma_C \subset [\tfrac{1}{2},1] \times M$ whose closure $\overline{\Gamma}_C \subset [\tfrac{1}{2},1] \times M$ is a smooth submanifold.

The open subset $f^{-1}(U_k) \subset \mathcal{G}$ consists of those germs consisting of a collection of critical points along with
\begin{enumerate}[(i)]
	\item either at most $k-1$ of each positive and negative transverse crossings,
	\item or $k$ of each positive and negative transverse crossings such that there is a unique pair that is geodesically closest.
\end{enumerate}
Define a homotopy $H_t$ on $f^{-1}(U_k)$ as follows. On $f^{-1}(C_{k-1})$ it is the identity; for $ x \in f^{-1}(U_k) \setminus f^{-1}(C_{k-1}^\pm(M))$ for $t \in [0, \tfrac{1}{2}]$ push the critical points either left or right (if they are maxima and minima respectively) to remove them, and arrange the germs of the cancelling pair to be those of $\overline{\Gamma}_{f(x)}|{\{\tfrac{1}{2}\} \times M}$, and for $t \in [\tfrac{1}{2},1]$ take the germ of the manifold $\overline{\Gamma}_{f(x)}$ near $\{t\} \times M$.

Now Lemma 7.2 of \cite{IteratedLoopSpaces} applies to the data $(f, C_k^\pm(M), U_k, h_t, H_t)$ and implies that $f$ is a quasi-fibration with contractible fibres, and hence a weak homotopy equivalence. Hence by Lemma \ref{lem:Germs},

\begin{thm}\label{EquivalenceOfCategoriesForConfigurationSpaces}
The map $\Psi_\theta^c(\bR \times M)^i \to C^\pm(M)$ that intersects an oriented 1-submanifold with $\{0\} \times M$ and discards critical points is a weak homotopy equivalence. Hence the functor $F$ induces an equivalence $B\mathcal{C}_1^+(M) \simeq C^\pm(M)$.
\end{thm}

Applying Theorem \ref{MainTheorem} now implies that
$$C^\pm(M) \simeq \Gamma_c(T_{SO(1)}^{fib}(TM) \to M)$$
which is McDuff's Theorem 1.3, after noting that the space she denotes $X_n$ is $\Th(TS^n \to S^n) \simeq T_{SO(1)}(\bR^n)$.

\begin{cor}
By the discussion in \S\ref{sec:HomotopyInvariance}, the best approximation by a homotopy invariant functor to $C^\pm(-)$ is $\Omega^{\infty-1} \MT{SO}{1} \wedge -_+ \simeq Q(-_+)$, i.e.\ stable homotopy theory.
\end{cor}

One may perform a similar analysis for the spaces $C^\pm(M;X_+)$ of signed configurations in $M$ labelled by a space $X_+$ with disjoint basepoint. The tangential structure in this case is $\theta: \Gr_1^+(\bR^\infty) \times X \to \Gr_1(\bR^\infty)$ and we can prove
$$C^\pm(M;X) \simeq B\mathcal{C}_\theta(M) \simeq \Gamma_c(T^{fib}_{SO(1) \times X}(TM) \to M)$$
where $T_{SO(1) \times X}(\bR^n) \simeq \Th(TS^n \to S^n) \wedge X_+$. Thus the best approximation by a homotopy invariant functor to $C^\pm(-;X)$ is $Q(X_+ \wedge -_+)$, the homology theory represented by the suspension spectrum of $X_+$.

\subsubsection{Concluding remarks}
The discussion in this section suggests the following question: is there a space $C^{sing}_{d-1}(M)$ whose points are oriented $(d-1)$-dimensional \textit{singular} submanifolds of $M$, topologised in such a way that singularities can be resolved continuously in this space, and such that the natural functor
$$C_d^+(M)^i \lra \mathrm{Path}(C^{sing}_{d-1}(M))$$
gives an equivalence of classifying spaces? The singularities involved should be precisely those which occur as level sets of smooth functions with isolated critical points, so the set $C^{sing}_{d-1}(M)$ is essentially determined. Such a space would have its homotopy type described by Theorem \ref{MainTheorem}.

One may then ask if it admits a compatible smooth structure; if so the locus of singular manifolds ought to have infinite codimension, and its complement should be $\coprod_{[X]} \Emb(X, M)/\Diff^+(X)$ where the disjoint union is over diffeomorphism classes of $(d-1)$-dimensional non-singular submanifolds of $M$. Thus $C^{sing}_{d-1}(M)$ could be interpreted as a partial compactification of this space of smooth submanifolds of $M$.

Such a discussion should be related to the work of R. Sadykov \cite{SadykovStableCharClasses} and his notion of marked fold maps.

\section{Group completion of the configuration-space monoid}\label{sec:ConfigurationSpaceMonoid}

We wish to consider the category $\mathcal{C}^X_0(M)$, which has objects $-1$-submanifolds of $\{ a \} \times M$, so simply a copy of $\mathbb{R}$.  A morphism from $a$ to $b$ is a finite configuration of points in $(a, b) \times M$ labelled by points of $X$. This is simply the cobordism category corresponding to the tangential structure $\theta : \Gr_0(\bR^\infty) \times X \to \Gr_0(\bR^\infty)$. In the spirit of Segal \cite{SegalConfigurationSpace} we define the following configuration spaces.

\begin{defn}
Let $\tilde{C}_k(Y)$ be the space of \textit{ordered configurations of cardinality $k$} in $Y$, topologised as a subset of $Y^k$. Let $C(Y ; X)$ be the space of \textit{finite configurations} in $Y$ labelled by $X$, topologised as the disjoint union $\coprod \tilde{C}_k(Y) \times_{\Sigma_k} X^k$.

If $Y$ is of the form $M \times \mathbb{R}$, let $C'(M \times \mathbb{R} ; X)$ be the space of pairs $\{ (c, t) \in C(M \times \mathbb{R} ; X) \times \mathbb{R}^+ | c \subset (0, t) \}$, topologised as a subspace of $C(M \times \mathbb{R} ; X) \times \mathbb{R}^+$. This is a topological monoid under juxtaposition and is homotopy equivalent to $C(M \times \mathbb{R} ; X)$, as in \cite{SegalConfigurationSpace}.
\end{defn}

If we consider $C'(M \times \mathbb{R} ; X)$ as a topological category, there is a continuous functor
$$F : \mathcal{C}^X_0(M) \lra C'(M \times \mathbb{R} ; X)$$
that sends a morphism $C \subset (a,b) \times M$ labelled by $X$ to the translated configuration $C - a \subset (0, b-a) \times M$ with the same labelling. This is clearly functorial and is continuous in the topology given. One may prove the analogue of Theorem \ref{EquivalenceOfCategoriesForConfigurationSpaces} for this functor directly.

\begin{thm}
The functor $F$ induces a weak homotopy equivalence on classifying spaces, $BF: B \mathcal{C}_0^X(M) \overset{\simeq}\to B C'(M \times \mathbb{R} ; X)$.
\end{thm}

For a vector bundle $V$, let $V^+$ denote the bundle of pointed spheres obtained by taking the fibrewise one-point compactification, and let $V^+ \wedge_{fib} Y$ denote the space obtained by taking the fibrewise smash product with a pointed space $Y$. Theorem \ref{MainTheorem} establishes the weak homotopy equivalence
$$B \mathcal{C}^X_0(M) \simeq \Gamma_c ((TM \oplus \epsilon^1)^+ \wedge_{fib} (X_+) \to M),$$
so in particular we have
$$B \mathcal{C}^X_0(\mathbb{R}^{n-1}) \simeq \Omega^{n-1} S^{n} \wedge X_+.$$
Thus there is a weak homotopy equivalence
$$\Omega B C'(\mathbb{R}^n ; X) \simeq \Omega^{n} \Sigma^{n} X_+$$
which identifies the group-completion of the monoid $C'(\mathbb{R}^n ; X)$ as $\Omega^{n} \Sigma^{n} X_+$.

The discussion in \S \ref{sec:HomotopyInvariance} shows that the best approximation to $B\mathcal{C}_0^X(-) \simeq BC'(-;X)$ by a homotopy invariant functor is $Q(\Sigma -_+ \wedge X_+)$.

\begin{rem}
In \cite{SegalConfigurationSpace}, using the notation from that paper, Segal shows that $\Omega B C'_n(X) \simeq \Omega^n \Sigma^n X$, and there is no additional basepoint. The difference is that Segal allows labels in $X$ to move to some basepoint and the point vanish from a configuration, and we do not. In his framework this simply corresponds to making the basepoint be a disjoint point, and then we recover our notion of configuration spaces.
\end{rem}

\section{Group completion of the braid monoid}\label{sec:BraidMonoid}
In this section we outline an application of the methods of this paper that does not fit into the framework of Theorem \ref{MainTheorem}, but can be proved without difficulty by similar means.

Let us write $\mathcal{C}^{Br}_{1, 3}$ for the cobordism category of braids, that is the category with objects configurations in $\{ a \} \times \mathbb{R}^2$ and with morphism oriented 1-manifolds whose tangent vectors always lie in the forwards hemisphere of $S^2 = \Gr_1^+(\bR^3)$. We topologise this as a subcategory of $\mathcal{C}_1^+(\bR^2)$.

We are also interested in the cobordism category with morphisms oriented 1-manifolds in $\mathbb{R}^3$ with the following structure. An oriented 1-manifold $W \subseteq \mathbb{R}^3$ determines a Gauss map $\tau_W : W \to \Gr^+_1(\mathbb{R}^3)$, and this Grassmannian has a pathspace fibration $\theta : P\Gr^+_1(\mathbb{R}^3) \to \Gr^+_1(\mathbb{R}^3) = S^2$ of paths starting in the forwards hemisphere. Equip $W$ with a lift $\ell$ of $\tau_W$ up this fibration, and call such a lift a \textit{path structure} on $TW$.

Let the category of homotopy braids, $C^{hBr}_{1,3}$, be the cobordism category defined using this structure, with objects configurations in $\{ a \} \times \mathbb{R}^2$ with the forwards orientation of the line bundle over them and a path structure on this oriented line bundle, and morphisms collared 1-manifolds in $[a, b] \times \mathbb{R}^2$ with a path structure that agrees with the given path structures at $a$ and $b$.

\vspace{2ex}

There is a continuous functor
$$F : \mathcal{C}^{Br}_{1, 3} \lra \mathcal{C}^{hBr}_{1, 3}$$
assigning to each point on a braid the path structure given by the constant path at its forwards tangent vector. It is well known that $B\mathcal{C}^{Br}_{1, 3} \simeq \coprod_{n \geq 0} B \beta_n$, the disjoint union of the classifying spaces of the braid groups, which is a monoid under disjoint union of braids. On the other hand, applying the methods of this paper in this case shows that
$$B\mathcal{C}^{hBr}_{1, 3} \simeq \Gamma_c(\mathbb{R}^2 \times \Th(\theta^* (\gamma^+_1)^{\perp} \to P\Gr^+_1(\mathbb{R}^3)) \lra \mathbb{R}^2) \simeq \Omega^2 S^2$$
and so on classifying spaces we obtain a map
$$\coprod_{n \geq 0} B \beta_n \lra \Omega^2 S^2.$$
It is known by work of Segal \cite{SegalConfigurationSpace} that the group completion of the braid monoid is $\Omega^2 S^2$. Below we show that this map \textit{is} the group-completion map.
\begin{thm}
This map is the group-completion of the braid monoid.
\end{thm}
\begin{proof}[Proof sketch]
It is enough to show that it is a map of algebras over the little 2-discs operad, as then delooping it twice we obtain a self-map of $S^2$ which is easily checked to be of degree 1. 

The action of the little 2-discs operad on $B\mathcal{C}^{hBr}_{1,3}$ and $B\mathcal{C}^{Br}_{1,3}$ comes from disjoint union in the category in both cases, which is clearly preserved by the functor $F$. The monoidal product on $\coprod_{n \geq 0} B \beta_n$ comes from disjoint union of braids also, so the homotopy equivalence $B\mathcal{C}^{Br}_{1,3} \simeq \coprod_{n \geq 0} B \beta_n$ is an equivalence of algebras over the little 2-discs operad. All that is left to show is that $B\mathcal{C}^{hBr}_{1,3} \simeq \Omega^2 S^2$ is an equivalence of algebras over the little 2-discs operad: that is, that the double-loop space structure on $B\mathcal{C}^{hBr}_{1,3}$ is the usual one.

It is useful to recall the intermediate spaces
$$B\mathcal{C}^{hBr}_{1,3} \simeq \Psi_{\theta}^{c}(\mathbb{R}^3) \simeq \Gamma_c(\mathbb{R}^2 \times \Th(\theta^* (\gamma^+_1)^{\perp} \to P\Gr^+_1(\mathbb{R}^3)) \lra \mathbb{R}^2) \simeq  \Omega^2 S^2$$
which also admit actions of the little 2-discs operad. On $\Psi_\theta(\mathbb{R}^3)$ the action is by choosing once and for all diffeomorphisms $\varphi: \mathbb{R}^2 \lra D^2$, then taking submanifolds of $\mathbb{R}^2 \times \mathbb{R}$ (where the last factor determines the forwards direction) to submanifolds of $D^2 \times \mathbb{R}$ and then composing using the little 2-discs operad is the obvious way. On $\Gamma_c(\mathbb{R}^2 \times \Th(\theta^* (\gamma^+_1)^{\perp} \to P\Gr^+_1(\mathbb{R}^3)) \lra \mathbb{R}^2)$ we use the same diffeomorphism $\varphi$ to get sections of the corresponding bundle over $D^2$, and the little 2-discs operad again acts in the obvious way, extending by the section $\emptyset$ where necessary.

It is now clear that the middle equivalence is a map of algebras over the little 2-discs operad, and one can also see that the last map is, using the standard action on $\Omega^2 S^2$. It remains to show that the action we described on $\Psi_{\theta}^{c}(\mathbb{R}^3)$ and the action coming from disjoint union in the category agree, but this is clear.
\end{proof}

The observation of this section leads to an interesting general question. There are many ``rigid" local structures that can be put on submanifolds (in the above example, insisting that their tangent vectors always lie in the forward hemisphere), and typically such structures admit a ``flexible" or ``homotopical" analogue (in the above example, asking for a path from each tangent vector into the forward hemisphere). Taking the cobordism categories of manifolds equipped with such structures, there is a functor $\mathcal{C}^{rig}_d \to \mathcal{C}^{flex}_d$, and in the above example we have shown that on classifying spaces this is precisely group completion. It is not hard to construct examples of rigid structures where $\mathcal{C}^{rig}_d$ is empty but $\mathcal{C}^{flex}_d$ is not, so there is not always such a fundamental relationship: under what conditions on the rigid local structure is there a close relationship between the two classifying spaces?